\documentclass[12pt]{amsart}

\usepackage{amsmath,amssymb,amsthm,xypic,epsfig }
\usepackage{enumerate}
\usepackage{color}
\usepackage{bbm}
\usepackage{appendix}
\usepackage[margin=30truemm]{geometry}
\usepackage[colorinlistoftodos]{todonotes}
\usepackage{version}
\usepackage{tikz}
\usepackage{enumitem}
\usepackage{setspace}
\usetikzlibrary{shapes.geometric, shapes.misc}
\DeclareFontFamily{OT1}{rsfs}{}
\DeclareFontShape{OT1}{rsfs}{n}{it}{<-> rsfs10}{}
\DeclareMathAlphabet{\curly}{OT1}{rsfs}{n}{it}

\newtheorem{Thm}{Theorem}[section]
\newtheorem{lem}[Thm]{Lemma}

\newtheorem{cor}[Thm]{Corollary}

\newtheorem{prop}[Thm]{Proposition}

\theoremstyle{remark}
\newtheorem{Rem}[Thm]{Remark}
\newtheorem{ex}[Thm]{Example}
\theoremstyle{definition}
\newtheorem{defn}[Thm]{Definition}

\newcommand{\Spec}{\mathop{\mathrm{Spec}}\nolimits}

\newcommand{\Hom}{\mathop{\mathrm{Hom}}\nolimits}
\newcommand{\Ext}{\mathop{\mathrm{Ext}}\nolimits}

\newcommand{\Proj}{\mathop{\mathrm{Proj}}\nolimits}

\newcommand{\ord}{\mathop{\mathrm{ord}}\nolimits}
\newcommand{\gr}{\mathop{\mathrm{gr}}\nolimits}

\newcommand{\sHom}{\mathop{\mathcal{H}om}\nolimits}

\newcommand{\Image}{\mathop{\mathrm{Im}}\nolimits}

\newcommand{\length}{\mathop{\mathrm{length}}\nolimits}

\newcommand{\Val}{\mathop{\mathrm{Val}}\nolimits}

\def\Sym{\operatorname{Sym}}

\def\rank{\operatorname{rank}}
\def\height{\operatorname{ht}}
\def\dim{\operatorname{dim}}

\def\Rees{\operatorname{Rees}}

\def\vol{\operatorname{vol}}
\def\deg{\operatorname{deg}}

\excludeversion{memo}
\allowdisplaybreaks
\begin{document}

\title[Algebraic tangent cone of torsion-free sheaves via valuations]{Optimal algebraic tangent cone of torsion--free sheaves via valuations}
\author{Yohei Hada}

\maketitle
\thispagestyle{empty}

\date{\today}

%\setcounter{tocdepth}{4} 
%\noindent
\tableofcontents

%%%abstract%%%
\begin{abstract}
We develop a valuation-theoretic framework for studying tangent cones of torsion--free sheaves on algebraic varieties. To analyze these objects, we introduce a slope stability theory, including the Harder--Narasimhan filtrations, for finitely generated $\mathbb{R}$--graded modules over finitely generated $\mathbb{R}_{\geq 0}$-graded algebras. Using it, we show that there is a canonically determined tangent cone of torsion--free sheaves, up to the expected equivalence ambiguity, for quasi-regular valuations, which generalize Chen--Sun \cite{CS_2018}.
\end{abstract}
%%%%%%%%%%

\section{Introduction}\label{sec:intro}
In this paper, we introduce slope stability of finitely generated $\mathbb{R}$-graded modules on finitely generated $\mathbb{R}_{\geq 0}$-graded algebras over an algebraically closed field $\mathbbm{k}$. Then we apply it to show the existence and uniqueness of optimal algebraic tangent cones of torsion-free sheaves, which we define, along quasi--regular valuations centred at a closed point on algebraic varieties. It extends the notion of the optimal algebraic tangent cone introduced by Chen--Sun \cite{CS_2018} far beyond the blow--up valuations, i.e. $\ord_E$ for $E$ is the exceptional divisor of blow--up at smooth point, treated in their work.

The starting point of our work is the series of deep results by Chen--Sun \cite{CS_2018},\cite{CS_20},\cite{CS_analitic_tc},\cite{CS_21} on local structures of reflexive sheaves equipped with admissible HYM connections on K\"{a}hler manifolds. The notion of admissible HYM connection is introduced by Bando--Siu\cite{Bando1994STABLESA}: 
\begin{defn}
    Let $(X,\omega)$ be a K\"{a}hler manifold and $E$ be a reflexive sheaf on $X$. Let $S\subset X$ be the non--locally--free locus of $E$.
    \begin{enumerate}
        \item An {\bf admissible metric} on $E$ is a metric $h$ on $E|_{X-S}$ such that for any $x\in X$, there is an open neighbourhood $U$ of $x$ with following properties:
        \begin{itemize}
            \item The curvature tensor $F_h$ is square integrable on $U-S$.
            \item The trace of the curvature tensor $\sqrt{-1}\Lambda F_h$ is bounded on $U-S$.
        \end{itemize}
        \item An admissible metric on $E$ is called {\bf admissible Hermitian--Einstein} (or admissible HE metric for short) if there is a constant $\lambda\in \mathbb{R}$ with $\sqrt{-1}\Lambda F_h=\lambda\cdot I_E$, where $I_E$ is the identity morphism of $E$.
        \item The Chern connection of an admissible HE metric is called {\bf admissible Hermitian--Yang--Mills connection} (or admissible HYM connection for short).
    \end{enumerate}
\end{defn}
Admissible metrics always exists for reflexive sheaves on K\"{a}hler manifolds. It is proven in Bando--Siu\cite{Bando1994STABLESA} that the slope polystability of reflexive sheaves is equivalent to the existence of an admissible HE metric (Kobayashi--Hitchin correspondence).

In Chen-Sun's work, they introduced the notion of the analytic tangent cone of an admissible HYM connection. Let $E$ be a reflexive sheaf on $B=B(1)=\{z\in \mathbb{C}^n\mid \|z\|<1\}$ which has only $1$ point singularity $0$, and $A$ be an admissible HYM connection of $E$. For $\lambda>0$, we have $\lambda: B(\lambda^{-1})\to B; z\mapsto \lambda\cdot z$ and $\lambda^*A=:A_\lambda$. By Uhlenbeck's compactness, we have a subsequence $(\lambda_i)$ and a sequence of gauge transforms $(g_i)$ such that $\lambda_i\to 0$ and $g_i^*A_{\lambda_i}$ converges to some HYM connection on $(\mathbb{C}^n-\{0\})-\Sigma$, where $\Sigma$ is an analytic subset of $\mathbb{C}^n-\{0\}$ invariant under $\mathbb{C}^*$-action and has locally finite Hausdorff measure of real codimension $4$. By the removable singularity theorem of Bando\cite{199161}, we have a reflexive sheaf $E_\infty$ and an admissible HYM connection $A_\infty$ on $(\mathbb{C}^n,\omega_0)$, where $\omega_0$ is the Euclidean metric on $\mathbb{C}^n$. By passing further subsequence, Yang--Mills energy of $A_\lambda$ converges to a Radon measure of type $|F_{A_\infty}|^2d\vol_{\omega_0}+8\pi^2\nu$, where $\nu$ is supported in $\Sigma\cup\{0\}$. $\nu$ induces a $\mathbb{C}^*$-equivariant cycle on $\mathbb{C}^n$, which is written as the cone of the codimension 2 cycle $\Sigma_b$ on $\mathbb{P}^{n-1}$. 

On the other hand, let $E$ be a reflexive sheaf on $B=B(1)$ and $\pi: \hat{B}\to B$ be a blow--up at $0$. Chen--Sun \cite{CS_2018} defined extensions of $E$ as a reflexive sheaf $\hat{E}$ on $\hat{B}$ with an isomorphism to $\pi^*E|_{\hat{B}-D}$ on $\hat{B}-D$, where $D\subset \hat{B}$ is the exceptional divisor, and called $\hat{E}|_D$ an algebraic tangent cone. Furthermore, they defined the functional $\Phi$ from all extensions as $\Phi(\hat{E})=\mu_{\max}(\hat{E}|_D)-\mu_{\min}(\hat{E}|_D)$, where $\mu$ is the slope of sheaves on the polarized variety $(D,N_{D/\hat{B}})$. Chen--Sun \cite{CS_2018} showed that there exists an extension $\hat{E}$ with $\Phi(\hat{E})\in [0,1)$, which is called an optimal extension, and the pull--back of the graded sheaf of HNF of optimal extensions $\gr^{HNF}(\hat{E}|_D)$ to the cone of $\mathbb{P}^{n-1}\cong D$ do not depend the choice of optimal extensions. 

Furthermore, Chen--Sun \cite{CS_21} showed that $E_\infty$ is in fact the reflexive hull of the pull--back to the cone of $D$ of the graded sheaf of Harder--Narasimhan and Jordan--H\"{o}lder filtration (or Harder--Narasimhan--Seshadri filtration) $\gr^{HNS}(\hat{E}|_D)$ of the optimal extension and $\Sigma_b$ is the cycle induced by the cokernel of the reflexive hull of the above sheaf.

The purpose of this paper is to extend the theory of algebraic tangent cones by Chen--Sun to a broader valuation--theoretic framework.

To any finitely generated valuation $v: R\to \mathbb{R}_{\geq 0}\cup \{\infty\}$ over a finitely generated integral domain $R$ over a field $\mathbbm{k}$, we can associate a multi--graded $\mathbb{G}_m^r$-equivariant degeneration $\pi: \mathfrak{X}\to T$ where $T$ is a toric variety of $\dim(T)=r$ and $r$ is the rank of $v$. Let $M(v)\subset\mathbb{R}$ be the additive subgroup generated by the sub-semigroup $\Gamma:=v(R-\{0\})\subset \mathbb{R}$. As an abelian group, $M(v)$ is isomorphic to $\mathbb{Z}^r$. Set $M(v)_{\mathbb{R}}=M(v)\otimes_\mathbb{Z}\mathbb{R}$ and $N(v)=\Hom_{\mathbb{Z}}(M(v),\mathbb{Z})$. Let $\sigma\subset N(v)_{\mathbb{R}}:=N(v)\otimes_{\mathbb{Z}}\mathbb{R}$ be the cone defined by 
\[
\sigma=\{\phi\in N(v)_{\mathbb{R}}\mid \left<\phi,\gamma\right>\geq 0, \forall \gamma\in \Gamma\}.
\]
Then the affine toric variety $T(\sigma)$ associated by $\sigma\subset N(v)$ is $\Spec(\mathbbm{k}[\sigma^\lor \cap M(v)])$. Note that $\sigma^\lor\cap M(v)$ is the saturation of subsemigroup $\Gamma\subset M(v)$, i.e. $\gamma\in \sigma^\lor \cap M(v)$ iff $\gamma\in M(v)$ and $k\gamma\in \Gamma$ for some $k\in \mathbb{Z}_{>0}$. Thus $\Gamma^s:=\sigma^\lor \cap M(v)$ can be viewed as a sub--semigroup of $\mathbb{R}_{\geq 0}$. Let $\mathcal{R}$ be a $\mathbbm{k}[\Gamma^s]$-algebra defined by
\[
\mathcal{R}:=\bigoplus_{\gamma\in M(v)}\mathfrak{a}_\gamma(v)t^{-\gamma},
\]
where $\mathfrak{a}_\gamma(v):=\{f\in R\mid v(f)\geq \gamma\}$. Then $\mathcal{R}$ is finitely generated $\mathbbm{k}[\Gamma^s]$-algebra and thus we have 
\[
\pi: \mathfrak{X}:=\Spec(\mathcal{R})\to T(\sigma).
\]
Now, $\mathfrak{X}$ has the natural $(\mathbb{G}_m)^r$-action, $\pi$ is $(\mathbb{G}_m)^r$-equivariant, and restriction of $\pi$ to $\pi^{-1}((\mathbb{G}_m)^r)$ is $(\mathbb{G}_m)^r$-equivariantly isomorphic to $\Pr_2:\Spec(R)\times (\mathbb{G}_m)^r\to (\mathbb{G}_m)^r$. Note that for $w\in N(v)\cap \sigma$, $w$ determines a family $\pi_w: \mathfrak{X}_w\to \mathbb{A}^1$ by defining
\[
\mathfrak{X}_w=\Spec\left(\bigoplus_{\gamma\in \Gamma^s}\mathfrak{a}_s(v)t^{-w(s)}\right)\to \mathbb{A}^1_t.
\]
There is the canonical element $\xi$ in $N(v)_{\mathbb{R}}$ defined by the inclusion $M(v)\to \mathbb{R}$. This vector $\xi$ is called the {\bf Reeb vector} of $v$. Note also that the pair $(\pi:\mathfrak{X}\to T(\sigma),\xi)$ completely determines $v$.

From this perspective, we consider degenerations of torsion-free sheaves on $X=\Spec(R)$ along finitely generated valuations centred at a closed point $x\in X$. For a torsion-free sheaf $E=M^\sim$ on $X$, we can pull back $E$ to $\mathfrak{X}|_{(\mathbb{G}_m)^r}$ via the composition of the natural isomorphism $\mathfrak{X}|_{(\mathbb{G}_m)^r}\cong X\times (\mathbb{G}_m)^r$ and the first projection. We define degenerations of $E$ along $v$ as $(\mathbb{G}_m)^r$-equivariant extensions of this sheaf to torsion-free sheaves on $\mathfrak{X}$. This definition is consistent with the definition of extensions of $E$ in \cite{CS_2018}. By the weight decomposition of the $(\mathbb{G}_m)^r$-action, it corresponds to the geometric $v$-valuation on $M$ defined in the Definition \ref{v_valuative_functions}. 

However, extension is not unique since we extend modules along divisors. Thus, we have a question: Is there a canonical extension for given modules along finitely generated valuations?

In the case of the blow--up valuation on smooth points, the above-mentioned optimal extension of Chen--Sun \cite{CS_2018} answers this question and \cite{CS_21} shows that it actually corresponds to the limit of rescalings of admissible HYM connections. In this paper, we extend the notion of optimal extensions to quasi--regular valuations centred at closed points on affine varieties and show that there is a unique optimal extension in the following form (see \ref{main_theorem}):
\begin{Thm}
    Let $R$ be an affine domain over an algebraically closed field $\mathbbm{k}$ and $v\in \Val_{x,X}$ be a quasi--regular valuation centred at a closed point $x\in X:=\Spec(R)$ with index $\delta$, i.e. $v(K(R)-\{0\})=\delta\cdot \mathbb{Z}$. For any torsion--free $R$--module $M$, there is a sheaf $\hat{E}$ on $\mathfrak{X}=\Spec(\Rees_v(R))$ extending $M^\sim$ on $\Spec(R)$ $\mathbb{G}_m$-equivariantly such that 
    \[
    \Phi(\hat{E}|_{\mathfrak{X}_0}):=\mu_{\max}(\hat{E}|_{\mathfrak{X}_0})-\mu_{\min}(\hat{E}|_{\mathfrak{X}_0})<\delta.
    \]
    Furthermore, the graded module of the Harder--Narasimhan filtration of $\hat{E}|_{\mathfrak{X}_0}$ satisfying the above inequality is unique up to change of grading of each direct summand, where the Harder-Narasimhan filtration of graded module is defined in \ref{HNF} and $\mu_{\max},\mu_{\min}$ are defined as the slope (defined in \ref{slope}) of $\underline{M_1}$ and $\underline{M}/\underline{M_{l-1}}$, $(\underline{M_i})_{i=1}^l$ is the Harder-Narasimhan filtration of $\underline{M}$, the graded module corresponding to the $\mathbb{G}_m$-equivariant sheaf $\hat{E}|_{\mathfrak{X}_0}$.
\end{Thm}

The paper is organized as follows. In Section 2, we prove the asymptotic Riemann--Roch type theorem for finitely generated $\mathbb{R}$-graded modules over finitely generated $\mathbb{R}_{\geq 0}$-graded algebras over $\mathbbm{k}$ and introduce the stability notion of these modules. In Section 3, we define the notion of $v$-valuation, which corresponds to (and generalize) the notion of extension in Chen--Sun \cite{CS_2018}, and prove the main theorem. In Section 4, we provide an explicit example that lies outside the original framework of Chen--Sun.

Throughout this paper, we fix a base field $\mathbbm{k}$, and all algebras, schemes, and morphisms are defined over 
$\mathbbm{k}$ unless otherwise stated. Also, we assume that $\mathbbm{k}$ is algebraically closed.
\section{Preliminaries}
Let $\Gamma\subset \mathbb{R}_{\geq 0}$ be a finitely generated sub-semigroup, and $R=\bigoplus_{\gamma\in \Gamma}R_\gamma$ be a finitely generated graded domain with $R_0=\mathbbm{k}$. Let $M=\bigoplus_{\lambda\in \mathbb{R}}M_\lambda$ is finitely generated $\mathbb{R}$-graded $R$-module. The goal of this section is to establish the stability notion of $M$, which is equivalent to the $\mu$-stability when $\Gamma=\mathbb{Z}_{\geq 0}$, $\Lambda=\mathbb{Z}$.

Before doing this, we define some notations: 
\begin{defn}
    Let $R,M$ be as above. For $\lambda\in \mathbb{R}$, define the new graded $R$-module $M(\lambda)$ as $M(\lambda)_t=M_{\lambda+t}$. 
\end{defn}
\subsection{Asymptotic Riemann-Roch theorem}
In this subsection, we establish a Riemann-Roch type formula for our setting. 
\begin{Thm}[Riemann-Roch]\label{generalized_RR}
Let $R=\bigoplus_{\gamma\in \Gamma}R_\gamma$ be a $n$-dimensional finitely generated domain over $\mathbbm{k}$, and $M=\bigoplus_{\lambda\in \Lambda}M_\lambda$ be a finitely generated graded $R$-module. There exists some constant $a_n$, $a_{n-1}\in \mathbb{R}$ which satisfies the following properties:
\begin{align*}
    &\dim(M_{\leq x})=\frac{a_n(M)}{n!}x^n+O(x^{n-1})\\
    &\lim_{T\to \infty}\frac{1}{T^n}\int_1^T\left(\dim(M_{\leq x})-\frac{a_n(M)}{n!}x^n-\frac{a_{n-1}(M)}{(n-1)!}x^{n-1}\right)dx=0,
\end{align*}
where $M_{\leq x}:=\bigoplus_{\lambda\leq x}M_\lambda$. Moreover, $a_n(M)=\rank(M)a_n(R)$.
\end{Thm}
Firstly, we prepare some fundamental lemmas. 
\begin{lem}\label{fundamental_lemma}
    Let $r$ be the rank of $M$. Then there is $\lambda_1,\dots,\lambda_r\in \Lambda$ and an injection 
    \[
    \bigoplus_{i=1}^r R(-\lambda_i)\to M
    \]
    whose cokernel is a torsion module.
\end{lem}
\begin{proof}
    Take a homogeneous generator $\{m_1,\dots,m_N\}$ of $M$. Then the $K(R)$-vector space $M\otimes_{R}K(R)$ is spanned by $\{m_i\otimes 1\}_{i=1}^N$. Thus there is some $i_1,\dots,i_r\in \{1,\dots,N\}$ such that $\{m_{i_1}\otimes 1,\dots,m_{i_r}\otimes 1\}$ is the basis of $M\otimes_RK(R)$. Let $\lambda_i:=\deg(m_i)$. Then we have the injection which satisfies the required properties.
\end{proof}
\begin{lem}\label{weak_riemann_roch_for_polynomial_ring}
    Let $\gamma_1,\dots,\gamma_n\in \mathbb{R}_{>0}$, and $R=\mathbbm{k}[x_1^{(\gamma_1)},\dots,x_n^{(\gamma_n)}]$ be the polynomial algebra, with $\deg(x_i)=\gamma_i$. Then 
    \[
    \dim(R_{\leq x})=\frac{1}{n!\prod_{i=1}^n\gamma_i}x^n+O(x^{n-1})
    \]
\end{lem}
\begin{proof}
    We prove this by induction on $n$. When $n=1$, 
    \[
    \dim(R_{\leq x})=\lfloor\frac{x}{\gamma_1}\rfloor+1. 
    \]
    Thus this is clear. Let us assume that ($n-1$)-case is proved. Then we have
    \begin{align*}
        \dim(R_{\leq x})=&\#\{(x_1,\dots,x_n)\in \mathbb{Z}_{\geq 0}^n\mid \sum_{i=1}^n \gamma_ix_i\leq x\}\\
        =&\sum_{k=0}^{\lfloor \frac{x}{\gamma_n}\rfloor}\dim(\mathbbm{k}[x_1^{(\gamma_1)},\dots,x_{n-1}^{(\gamma_n)}]_{\leq x-\gamma_n k})\\
        =&\sum_{k=0}^{\lfloor \frac{x}{\gamma_n}\rfloor}\frac{(x-\gamma_nk)^{n-1}}{(n-1)!\gamma_1\dots\gamma_{n-1}}+O(x^{n-1})\\
        =&\sum_{k=0}^{\lfloor \frac{x}{\gamma_n}\rfloor}\frac{\left\{\gamma_nk+(x-\gamma_n\lfloor\frac{x}{\gamma_n}\rfloor)\right\}^{n-1}}{(n-1)!\gamma_1\dots\gamma_{n-1}}+O(x^{n-1})\\
        =&\frac{x^n}{n!\gamma_1\dots\gamma_n}+O(x^{n-1})
    \end{align*}
\end{proof}
\begin{prop}\label{the_weak_riemann_roch_formula}
    Let $n:=\dim(R)$. Then, for any finitely generated $\mathbb{R}$-graded module $M$, there exists some constant $a_n(M)\geq 0$ such that 
    \[
    \dim(M_{\leq x})=\frac{a_n(M)}{n!}x^n+O(x^{n-1}).
    \]
    Moreover, $a_n(M)=\rank(M)a_n(R)$.
\end{prop}
\begin{proof}
    Let $(\alpha_n)$ be the statement of the proposition with $R=M$, $\dim(R)=n$ case, and $(\beta_n)$ be the statement of the proposition with $\dim(R)=n$. $(\alpha_0)$ and $(\beta_0)$ is trivial.
    \begin{enumerate}
        \item[\underline{Step 1}] Firstly, we prove that $(\alpha_n)\land (\beta_{n-1})$ implies $(\beta_n)$. Assume that $(\alpha_n)$ and $(\beta_{n-1})$ are proved and $\dim(R)=n$. Take an exact sequence 
        \[
        0\to \bigoplus_{i=1}^r R(-\lambda_i)\to M\to Q\to 0
        \]
        as in the Lemma \ref{fundamental_lemma}. Then $Q$ is supported in codimension $\geq 1$ subscheme of $\Spec(R)$. Using the induction on the number of associated primes of $Q$ with the assumption $(\beta_{n-1})$, we can prove $\dim(Q_{\leq x})=O(x^{n-1})$. In fact, if $\mathrm{Ass}(Q)=\{P\}$, $Q$ is $P$ primary, and we have Jordan-H\"{o}lder sequence $0=Q_0\subset Q_1\subset \dots\subset Q_l$, where $l=\length(Q)$. Each quotient $Q_i/Q_{i-1}$ is isomorphic to $(R/P)(\lambda_i)$ as graded $(R/P)$-module. Thus we have by assumption $(\beta_{n-1})$ that 
        \[
        \dim (Q_{\leq x})=\sum_{i}\dim(Q_i/Q_{i-1})_{\leq x}=O(x^{n-1}).
        \]
        Moreover, we get by this argument that
        \[
        \dim(Q_{\leq x})=\left(\sum_{P}\frac{\length(H^0_P(Q))a_{n-1}(R/P)}{(n-1)!}\right)x^{n-1}+O(x^{n-2}),
        \]
        where the sum runs over every associated prime of $Q$ with $\height(P)=1$.
        Therefore, 
        \[
        \dim(M_{\leq x})=\sum_{i=1}^r\dim(R_{\leq x-\lambda_i})=\frac{\rank(M)a_n(R)}{n!}x^n+O(x^{n-1}).
        \]
        \item[\underline{Step 2}]Secondly, we prove that $(\beta_{n-1})$ implies $(\alpha_n)$. By Noether's normalization theorem for graded algebras, we have some inclusion of graded $\mathbbm{k}$-algebras $S:=k[x_1^{(\gamma_i)},\dots,x_n^{(\gamma_n)}]\subset R$ such that $R$ is finitely generated graded $S$-module. Thus, by Lemma \ref{fundamental_lemma}, we get the exact sequence 
        \[
        0\to \bigoplus_{i=1}^r S(-\gamma_i)\to R\to Q\to 0
        \]
        of graded $S$-modules, where $Q$ is a torsion $S$-module. The argument as in \underline{Step 1} shows that 
        \[
        \dim(Q_{\leq x})=O(x^{n-1}).
        \]
        Therefore, using the Lemma \ref{weak_riemann_roch_for_polynomial_ring}, conclude that 
        \[
        \dim(R_{\leq x})=\frac{r}{n!\gamma_1\dots\gamma_n}x^n+O(x^{n-1}).
        \]
    \end{enumerate}
\end{proof}
In our setting, we can associate a valuation $v$ of $R$ in the natural way, namely, $v(\sum_\gamma f_\gamma)=\min\{\gamma\in \Gamma\mid f_\gamma\neq 0\}$, or $v(f)=\infty$ if $f=0$. Then, $a_n(R)$ is nothing but $\vol(v)$. 
\begin{defn}
    We call $a_n(R)$ as the {\bf volume} of the graded algebra $R$, and we sometimes write this as $\vol(R)$.
\end{defn}
Now, we can prove the first step of the theorem \ref{generalized_RR}. 
\begin{Thm}
    Let $\gamma_1,\dots,\gamma_n\in \mathbb{R}_{>0}$, and $S=\mathbbm{k}[x_1^{(\gamma_1)},\dots,x_n^{(\gamma_n)}]$. Then we have
    \[
    \lim_{T\to \infty}\frac{1}{T^n}\int_0^T \left(\dim(S_{\leq x})-\frac{x^n}{n!\gamma_1\dots\gamma_n}-\frac{(\sum_{i=1}^n \gamma_i)x^{n-1}}{2(n-1)!\prod_{i=1}^n \gamma_i}\right)dx=0
    \]
\end{Thm}
\begin{proof}
    We prove this theorem by induction on $n$. If $n=1$, we have 
    \[
    \dim(S_{\leq x})-\frac{x}{\gamma_1}=1-\left\{\frac{x}{\gamma_1}\right\},
    \]
    where $\{\lambda\}$ is the fractional part of $\lambda$. Thus,
    \begin{align*}
        \frac{1}{T}\int_0^T\left(\dim(S_{\leq x})-\frac{x}{\gamma_1}\right)dx=&\frac{1}{T}\sum_{k=0}^{\lfloor\frac{T}{\gamma_1}\rfloor-1}\int_{\gamma_1k}^{\gamma_1(k+1)}\left(k+1-\frac{x}{\gamma_1}\right)dx+o(1)\\
        =&\frac{1}{T}\frac{\gamma_1}{2}\cdot \lfloor\frac{T}{\gamma_1}\rfloor+o(1)\\
        \to& \ \frac{1}{2} \qquad (T\to \infty)
    \end{align*}
    Assume that the proposition is proved for $(n-1)$-dimensional $S$. Let $S'=\mathbbm{k}[x_1^{(\gamma_1)},\dots,x_{n-1}^{(\gamma_{n-1})}]\subset S$ be the graded subalgebra of $S$ generated by $x_1,\dots,x_{n-1}$. Then we have 
    \[
    \dim(S_{\leq x})=\sum_{k=0}^{\lfloor \frac{x}{\gamma_n}\rfloor}\dim(S'_{\leq x-\lambda_nk}).
    \]
    By the inductive hypothesis, 
    \begin{samepage}
        \begin{align*}
        \dim(S_{\leq x})-\frac{x^n}{n!\gamma_1\dots\gamma_n}=&\sum_{k=0}^{\lfloor \frac{x}{\gamma_n}\rfloor}\left(\dim(S'_{\leq x-\lambda_nk})-\frac{(x-\gamma_nk)^{n-1}}{(n-1)!\gamma_1\dots\gamma_{n-1}}\right)\\
        &+\sum_{k=0}^{\lfloor\frac{x}{\gamma_n}\rfloor}\frac{(x-\gamma_nk)^{n-1}}{(n-1)!\gamma_1\dots\gamma_{n-1}}-\frac{x^n}{n!\gamma_1\dots\gamma_n}\\
        =&\sum_{k=0}^{\lfloor\frac{x}{\gamma_n}\rfloor}\left(\frac{(\sum_{i=1}^{n-1}\gamma_i)(x-\gamma_nk)^{n-2}}{2(n-2)!\gamma_1\dots\gamma_{n-1}}+R(x-\gamma_nk)\right)\\
        &+\sum_{k=0}^{\lfloor\frac{x}{\gamma_n}\rfloor}\frac{(x-\gamma_nk)^{n-1}}{(n-1)!\gamma_1\dots\gamma_{n-1}}-\frac{x^n}{n!\gamma_1\dots\gamma_n}\\
        =&\frac{(\sum_{i=1}^{n-1}\gamma_i)x^{n-1}}{2(n-1)!\gamma_1\dots\gamma_n}+\sum_{k=0}^{\lfloor\frac{x}{\gamma_n}\rfloor}R(x-\gamma_nk)\\
        &+\sum_{k=0}^{\lfloor\frac{x}{\gamma_n}\rfloor}\frac{(x-\gamma_nk)^{n-1}}{(n-1)!\gamma_1\dots\gamma_{n-1}}-\frac{x^n}{n!\gamma_1\dots\gamma_n}+O(x^{n-2}),
    \end{align*}
    \end{samepage}
    where the function $R(x)$ is a function satisfying the following properties:
    \begin{enumerate}
        \item $R(x)=O(x^{n-2})$
        \item $\lim_{T\to \infty}\frac{1}{T^{n-1}}\int_0^T R(x)dx=0$.
    \end{enumerate}
    By Euler-Maclaurin's summation principle, we have 
    \[
    \int_0^{\frac{x}{\gamma_n}}\left(t-\lfloor t\rfloor -\frac{1}{2}\right)\frac{d}{dt}(x-\gamma_nt)^{n-1}dt=\sum_{k=1}^{\lfloor\frac{x}{\gamma_n}\rfloor-1}(x-\gamma_nk)^{n-1}+\frac{1}{2}x^{n-1}-\frac{x^n}{n\gamma_n}.
    \]
    Moreover, we have
    \begin{align*}
        &\frac{1}{x^n}\int_0^{\frac{x}{\gamma_n}}\left(t-\lfloor t\rfloor -\frac{1}{2}\right)\frac{d}{dt}(x-\gamma_nt)^{n-1}dt\\
        =&-\frac{(n-1)\gamma_n}{x^n}\int_0^{\frac{x}{\gamma_n}}\left(t-\lfloor t\rfloor-\frac{1}{2}\right)(x-\gamma_nt)^{n-2}dt\\
        =&-(n-1)\int_0^1\left(\frac{xu}{\gamma_n}-\left\lfloor\frac{xu}{\gamma_n}\right\rfloor-\frac{1}{2}\right)u^{n-2}du\\
        \to&\ 0\qquad (x\to \infty).
    \end{align*}
    We used Riemann-Lebesgue theorem in the last line. Thus, we can write 
    \[
    \sum_{k=0}^{\lfloor\frac{x}{\gamma_n}\rfloor}\frac{(x-\gamma_nk)^{n-1}}{(n-1)!\gamma_1\dots\gamma_{n-1}}-\frac{x^n}{n!\gamma_1\dots\gamma_n}=\frac{x^{n-1}}{2(n-1)!\gamma_1\dots\gamma_{n-1}}+r(x),
    \]
    where $r(x)=O(x^{n-1})$ and 
    \[
    \lim_{T\to \infty}\frac{1}{T^n}\int_0^T r(x)dx=0.
    \]
    Therefore, we get
    \[
    \dim(S_{\leq x})-\frac{x^n}{n!\gamma_1\dots\gamma_n}=\frac{(\sum_{i=1}^n\gamma_i)x^{n-1}}{2(n-1)!\gamma_1\dots\gamma_n}+r(x)+\sum_{k=0}^{\lfloor\frac{x}{\gamma_n}\rfloor}R(x-\gamma_nk)+O(x^{n-2})
    \]
    Finally, we have to estimate the sum $\sum_{k=0}^{\lfloor\frac{x}{\gamma_n}\rfloor}R(x-\gamma_nk)$. 
    \begin{align*}
        \left|\frac{1}{T^n}\int_0^T\sum_{k=0}^{\lfloor\frac{x}{\gamma_n}\rfloor}R(x-\gamma_nk)dx\right|&=\left|\frac{1}{T^n}\sum_{k=0}^{\lfloor\frac{x}{\gamma_n}\rfloor}\int_{\gamma_nk}^TR(x-\gamma_nk)dx\right|\\
        =&\left|\frac{1}{T}\sum_{k=0}^{\lfloor\frac{x}{\gamma_n}\rfloor}\frac{1}{T^{n-1}}\int_0^{T-\gamma_nk}R(x)dx\right|\\
        =&\left|\frac{1}{T}\sum_{k=0}^{\lfloor\frac{x}{\gamma_n}\rfloor}\left(1-\frac{\gamma_nk}{T}\right)^{n-1}\cdot \frac{1}{(T-\gamma_nk)^{n-1}}\int_0^{T-\gamma_nk}R(x)dx\right|
    \end{align*}
    Since $R(x)=O(x^{n-2})$, the integral $|\frac{1}{T^{n-1}}\int_0^T R(x)dx|$ is bounded by some constant $C>0$. Take an arbitrary $\epsilon>0$, then there is some $\delta>0$ that 
    \[
    \left|\frac{1}{T^{n-1}}\int_0^T R(x)dx\right|<\epsilon
    \]
    for any $T>\delta$. Thus, for $T\gg 0$, we get 
    \begin{align*}
        \left|\frac{1}{T^n}\int_0^T\sum_{k=0}^{\lfloor\frac{x}{\gamma_n}\rfloor}R(x-\gamma_nk)dx\right|\leq &\ \frac{C}{T}\sum_{\frac{T-\delta}{\gamma_n}<k\leq \frac{T}{\gamma_n}}\left(1-\frac{\gamma_nk}{T}\right)^{n-1}+\frac{\epsilon}{\gamma_n}\cdot \left(1-\frac{\delta}{T}\right)\\
        \to &\ \frac{\epsilon}{\gamma_n}\qquad(T\to \infty)
    \end{align*}
    Now, $\epsilon>0$ is arbitrary. Therefore, we have the desired consequence.
\end{proof}
\begin{Rem}
The argument above shows that the averaging procedure
naturally eliminates the oscillatory part of the counting function
even when the weights $\gamma_i$ are irrational.
In particular, $\dim(S_{\leq x})$ is generally neither polynomial nor
quasi-polynomial in $x$, yet the averaged coefficients $a_n$ and
$a_{n-1}$ remain well defined and depend continuously on the grading
weights $(\gamma_i)$.
This observation allows us to extend the asymptotic Riemann--Roch
formula beyond the rational case and, more conceptually, shows that
Ces\`{a}ro averaging provides a canonical regularization of asymptotic
growth for general finitely generated $\mathbb{R}_{\geq 0}$-graded algebras. Note also that 
\[
a_{n-1}(S)=\frac{1}{2}\vol(v)A(v),
\]
where $v$ is the induced valuation of the graded algebra $S$.
\end{Rem}
\begin{proof}[Proof of Theorem \ref{generalized_RR}]
    By Noether's normalization, it suffices to prove in the case $R=\mathbbm{k}[x_1^{(\gamma_1)},\dots,x_n^{(\gamma_n)}]$. Let 
    \[
    0\to \bigoplus_{i=1}^r R(-\lambda_i)\to M\to Q\to 0
    \]
    be as in Lemma \ref{fundamental_lemma}. Then we have 
    \begin{align*}
        \dim(M_{\leq x})=&\sum_{i=1}^r \dim(R_{\leq x-\lambda_i})+\dim(Q_{\leq x})\\
        =&\frac{\rank(M)a_n(R)}{n!}x^n-\frac{(\sum_{i=1}^r \lambda_i)a_n(R)-\rank(M)a_{n-1}(R)}{(n-1)!}x^{n-1}\\
        +&\left(\sum_{P}\frac{\length(H^0_P(Q))a_{n-1}(R/P)}{(n-1)!}\right)x^{n-1}+r(x)
    \end{align*}
    where the sum in the third term runs over all associated primes of $Q$ of height $1$.
    \[
    \lim_{T\to \infty}\frac{1}{T^n}\int_0^T r(x)dx=0.
    \]
    Put $a_{n-1}(M)$ as 
    \[
    a_{n-1}(M):=\rank(M)a_{n-1}(R)+\sum_{P}\length(H^0_P(Q))a_{n-1}(R/P)-\left(\sum_{i=1}^r \lambda_i\right)a_n(R),
    \]
    then we get the desired consequence.
\end{proof}

\begin{ex}[$\mathbb{Z}_{\geq 0}$-graded case]
    Let $R$ be a graded normal affine domain. Assume that $R$ is generated by grade $1$, i.e. $R$ can be written as $R=\mathbbm{k}[x_1,\dots,x_N]/I$, where all of $x_i$ are degree $1$ and $I$ is a homogeneous ideal of $k[x_1,\dots,x_n]$. 
    
    Let $X:=\Spec(R)$ and $0\in X$ be the point corresponding to the ideal $\mathfrak{m}_0=(x_1,\dots,x_n)$. Let $\pi: \hat{X}\to X$ be the blow--up at $0$. Then the exceptional divisor is exactly $E=\Proj(R)$ and the conormal bundle is $\mathcal{O}_E(1)$. By Riemann-Roch, 
    \[
    h^0(\mathcal{O}_E(m))=\frac{(-E)^{n-1}\cdot E}{(n-1)!}m^{n-1}-\frac{K_E\cdot (-E)^{n-2}}{2(n-2)!}m^{n-2}+O(m^{n-3}).
    \]
    Now, we can write  
    \[
    K_E+D=K_{\hat{X}}+E|E,
    \]
    where $D$ is an effective $\mathbb{Q}$-divisor on $E$ called different. Thus,
    \[
    -\frac{1}{2}K_E\cdot (-E|_E)^{n-2}=\frac{1}{2}(K_{\hat{X}}+E)\cdot (-E)^{n-1}+\frac{1}{2}D\cdot (-E|_E)^{n-2}=\frac{1}{2}A(E)(-E)^{n-1}\cdot E+\frac{1}{2}D\cdot (-E|_E)^{n-2}.
    \]
    Thus, 
    \begin{align*}
        \dim(R_{\leq x})=&\sum_{m=1}^{\lfloor x\rfloor}h^0(\mathcal{O}_E(m))\\
        =&\frac{(-E)^{n-1}\cdot E}{n!}\lfloor{x}\rfloor^n+\frac{(A(E)+1)(-E)^{n-1}\cdot E+D\cdot (-E|_E)^{n-2}}{2(n-1)!}\lfloor{x}\rfloor^{n-1}+O(x^{n-2}).
    \end{align*}
    Thus, letting $\frac{D\cdot (-E|_E)^{n-2}}{(-E)^{n-1}\cdot E}=\delta$, we have  
    \[
    \frac{1}{T^n}\int_0^T \left(\dim(R_{\leq x})-\frac{(-E)^{n-1}\cdot E}{n!}x^n-\frac{(A(E)+\delta)(-E)^{n-1}\cdot E}{2(n-1)!}x^{n-1}\right)dx=0.
    \]
    Thus, in this case, we have 
    \[
    \frac{a_{n-1}(R)}{a_n(R)}=\frac{1}{2}(A(E)+\delta),
    \]
    and $\delta$ measures how bad singularities on $E\subset \hat{X}$ are.
\end{ex}

\subsection{Slope stability}
In this subsection, we introduce the notion of slope stability of finitely generated torsion--free graded modules over a finitely generated $\mathbb{R}_{\geq 0}$-graded domain over $\mathbbm{k}$. 
\begin{defn}\label{definition_of_degree}
    Let $R=\bigoplus_{\gamma\in \mathbb{R}_{\geq 0}}R_\gamma$ be a finitely generated graded domain over $R_0=\mathbbm{k}$. For a finitely generated $\mathbb{R}$-graded $R$-module $M=\bigoplus_{\lambda\in \mathbb{R}}M_\lambda$, define the degree of $M$ as follows 
    \[
    \deg(M)=a_{n-1}(M)-\rank(M)a_{n-1}(R).
    \]
\end{defn}
\begin{prop}\label{additivity_of_degree}
    Let $R$ be a finitely generated $\mathbb{R}_{\geq 0}$-graded domain over $R_0=\mathbbm{k}$, and $M_1,M_2,M_3$ be finitely generated $\mathbb{R}$-graded $R$-modules. If there is a short exact sequence 
    \[
    0\to M_1\to M_2\to M_3\to 0
    \]
    of graded $R$-modules, we have
    \[
    \deg(M_2)=\deg(M_1)+\deg(M_3),
    \]
    i.e. $\deg$ extends to an additive function on the $K_0$-group of the abelian category of $\mathbb{R}$-graded $R$-modules.
\end{prop}
\begin{proof}
    It suffices to show that $a_{n-1}$ is additive. It is clear since
    \[
    \dim(M_{2,\leq x})=\dim(M_{1,\leq x})+\dim(M_{3,\leq x})
    \]
\end{proof}
\begin{defn}\label{slope}
    Let $R,M$ as in the Definition \ref{definition_of_degree}. Assume further that $a_n(M)\neq 0$. Define the slope $\mu(M)$ of $M$ as 
    \[
    \mu(M):=\frac{\deg(M)}{a_n(M)}.
    \]
\end{defn}
\begin{defn}\label{def_of_stability}
    Let $R$, $M$ as in the Definition \ref{definition_of_degree}, and assume further that $M$ is torsion--free. Then,
    \begin{enumerate}
        \item $M$ is {\bf semistable} if $\mu(N)\leq \mu(M)$ for any nonzero submodules $N$ of $M$.
        \item $M$ is {\bf stable} if $\mu(N)<\mu(M)$ for any non-trivial submodules $N$ of $M$.
    \end{enumerate}
\end{defn}

We can prove Harder--Narasimhan filtration in this setting:
\begin{Thm}[Harder--Narasimhan filtration]\label{HNF}
    Let $R$ be an $\mathbb{R}_{\geq 0}$-graded domain finitely generated over $R_0=\mathbbm{k}$. Let $M$ be a finitely generated $\mathbb{R}$-graded $R$-module. Then, there exists the unique filtration 
    \[
    0=M_0\subset M_1\subset \dots\subset M_l=M
    \]
    by graded submodules such that 
    \begin{enumerate}
        \item Each of $M_i/M_{i-1}$ is torsion free and semistable.
        \item $\mu(M_1)>\mu(M_2/M_1)>\dots>\mu(M_l/M_{l-1})$.
    \end{enumerate}
\end{Thm}
This theorem follows from a standard argument (cf. \cite{Huybrechts_Lehn_2010}). Before proving the theorem, we note the following proposition.
\begin{prop}\label{sea-saw-lemma}
    Let $R$ be an $\mathbb{R}_{\geq 0}$-graded domain finitely generated over $R_0=\mathbbm{k}$. Let $M_1,M_2,M_3$ be finitely generated $\mathbb{R}$-graded $R$-modules of $\rank\geq 1$. If we have a short exact sequence
    \[
    0\to M_1\to M_2\to M_3\to 0
    \]
    of graded modules, then the following inequality holds:
    \[
    \min\{\mu(M_1),\mu(M_3)\}\leq \mu(M_2)\leq \max\{\mu(M_1),\mu(M_3)\}.
    \]
    Moreover, the inequalities are strict if $\mu(M_1)\neq \mu(M_3)$.
\end{prop}
\begin{proof}
    This is a well-known consequence since both $a_n$ and $\deg$ are linear by Proposition \ref{additivity_of_degree}. 
\end{proof}
\begin{cor}\label{morphism_between_semistable_modules}
    Let $R$ be an $\mathbb{R}_{\geq 0}$-graded domain finitely generated over $R_0=\mathbbm{k}$, and $M_1,M_2$ be two semistable $\mathbb{R}$-graded modules. If there exists a non-trivial morphism
    \[
    \phi: M_1\to M_2
    \]
    of $\mathbb{R}$-graded modules, we have $\mu(M_1)\leq \mu(M_2)$. 
\end{cor}
\begin{proof}
    Let $K$ be the kernel of $\phi$, and $Q$ be the image of $\phi$. Then we get a short exact sequence
    \[
    0\to K\to M_1\to Q\to 0
    \]
    of $\mathbb{R}$-graded modules. By Proposition \ref{sea-saw-lemma} and semistability of $M_1$, we have 
    \[
    \mu(K)\leq \mu(M_1)\leq \mu(Q).
    \]
    Now, $M_2$ is also semistable. Thus we have 
    \[
    \mu(M_1)\leq \mu(Q)\leq \mu(M_2)
    \]
    as desired.
\end{proof}
\begin{proof}[Proof of Theorem \ref{HNF}]
    Firstly, we prove the existence. Let $\Sigma$ be a set of all nonzero graded submodules of $M$. Define the order $\prec$ on $\Sigma$ as $N_1\prec N_2$ iff $N_1\subset N_2$ and $\mu(N_1)\leq \mu(N_2)$. By Zorn's lemma, we can take $\prec$-maximal submodules. Now, $a_n(N)=\rank(N)\cdot a_n(R)\in \mathbb{Z}_{\geq 0}\cdot a_n(R)$. Thus we can take a $\prec$-maximal submodule $M_1$ such that $a_n(M_1)$ is minimal among $\prec$-maximal submodules. We prove this submodule is maximal destabilizing submodule, i.e. $\mu(N)\leq \mu(M_1)$ for any $N\in \Sigma$ and equality holds only if $N\subset M_1$. 

    Suppose that there exists $N\in \Sigma$ such that $\mu(N)\geq \mu(M_1)$. We can assume that $N$ is $\prec$-maximal. The $a_n$-minimality of $M_1$ shows that $N=M_1$ or $M_1\not\subset N$. Assume $M_1\not\subset N$. Consider the following short exact sequence
    \[
    0\to N\cap M_1\to N\oplus M_1\to N+M_1\to 0.
    \]
    We get 
    \begin{align*}
        &a_n(N\cap M_1)(\mu(N\cap M_1)-\mu(N))
        \\
        =&a_n(M_1)(\mu(M_1)-\mu(N+M_1))-(a_n(N+M_1)-a_n(M_1))(\mu(N+M_1)-\mu(N)).
    \end{align*}
    By $\prec$-maximality of $N$ and $M_1$, we have $\mu(N+M_1)<\mu(N)$ and $\mu(N+M_1)<\mu(M_1)$. Therefore, we have
    \[
    \mu(N\cap M_1)>\mu(N)\geq \mu(M_1).
    \]

    Let $N'\in \Sigma$ be the $\prec$-maximal graded submodule of $M_1$ such that $N\cap M_1\prec N'$, and $N''$ be the $\prec$-maximal submodule of $M$ such that $N'\prec N''$. The $a_n$-minimality of $M_1$ implies $N''\not\subset M_1$ and $\prec$-maximality of $M_1$ implies $M_1\not\subset N''$. We have $N'\subset N''\cap M_1$ and thus the above argument shows that $\mu(N')\leq \mu(N'')<\mu(N''\cap M_1)$. This contradicts with $\prec$-maximality of $N'$ among submodules of $M_1$. 

    Note that $M_1$ is saturated graded semistable submodule since saturation increases slope and $M_1$ is destabilizing submodule. Thus, we have done the first step of the Harder--Narasimhan filtration. Take $M_2\subset M$ as the preimage of the destabilizing submodule of $M/M_1$. By the following short exact sequence
    \[
    0\to M_1\to M_2\to M_2/M_1\to 0,
    \]
    and the fact that $\mu(M_1)>\mu(M_2)$, we get $\mu(M_1)>\mu(M_2/M_1)$ by Proposition \ref{sea-saw-lemma}. Keeping in this manner, the Noetherian property of $M$ forces the termination of this process. Hence we get the existence of the desired filtration. 

    Secondly, we prove the uniqueness. Let $(M_i)_{i=0}^l$ and $(M'_j)_{j=0}^{l'}$ be two Harder--Narasimhan filtrations. Without loss of generality, we can assume that $\mu(M_1)\geq \mu(M'_1)$. Let $j\in \{0,1,\dots,l'\}$ be the minimal such that $M_1\subset M'_j$. Then we get a non-trivial morphism $M_1\to M'_j/M'_{j-1}$ of semistable modules. Thus, we get $\mu(M_1)\leq \mu(M'_j/M'_{j-1})$ by Corollary \ref{morphism_between_semistable_modules}. If $j>1$, then $\mu(M'_1)\leq \mu(M_1)\leq \mu(M'_j/M'_{j-1})<\mu(M'_1)$, contradiction. Thus, $M_1\subset M'_1$ and $\mu(M_1)=\mu(M'_1)$. Let $i\in \{0,1,\dots,l\}$ be the minimal such that $M'_1\subset M_i$. If $i>1$, then the same argument shows that $\mu(M'_1)<\mu(M_1)$, contradiction. Therefore, we have $M_1=M'_1$. Continuing this argument, we have $l=l'$ and $M_i=M'_i$ for any $i\in \{0,1,\dots,l\}$.
\end{proof}
Finally, we consider some relations between $\deg$ and dual modules. 
\begin{defn}
    Let $R=\bigoplus_{\gamma\geq 0}R_\gamma$ be a graded affine domain over $\mathbbm{k}=R_0$ and $M,N$ be finitely generated graded $R$-modules. Define the internal hom module $\Hom_R(M,N)$ as the graded $R$-module with 
    \[
    \Hom_R(M,N)_\lambda=\Hom_{R}(M(-\lambda),N).
    \]
    If $N=R$, we write $\Hom_R(M,R)=M^\lor$. Right derived functor of the functor $\Hom_R(M,-)$ is denoted as $\{\Ext_R^i(M,-)\}_{i\geq 0}$
\end{defn}
Note that $M^\lor$ coincides with the dual module in a non-graded sense since $R$ is Noetherian and $M$ is finitely generated. So the $\Ext_R^i$ composed with the forgetful functor from the category of $\mathbb{R}$-graded $R$-modules to the category of $R$-modules, denoted by $F$, also coincides with the composition of $\Ext$ and $F$. We have the following theorem:
\begin{Thm}\label{degree_and_duality}
    If $R$ is normal, we have
    \[
    \deg(M^\lor)=-\deg(M)
    \]
\end{Thm}
First, we show the following lemma: 
\begin{lem}
    Let $R$ be a normal domain, and $M$ be a finitely generated torsion module over $R$. Then we have 
    \[
    \length(H^0_P(\Ext^1_R(M,R)))=\length(H^0_P(M))
    \]
    for any height $1$-primes $P$.
\end{lem}
\begin{proof}
    Since $M$ is finite, we can localize at $P$, and assume that $R$ is DVR. Let us denote $P=(\pi)$. By the structure theorem of modules over DVR, we have
    \[
    M\cong \bigoplus_{i=1}^s R/(\pi^{m_i})
    \]
    for some $m_1,\dots,m_s\in \mathbb{Z}_{>0}$. Thus, we can assume that $M=R/(\pi^m)$ for some $m>0$. By the exact sequence 
    \[
    0\to R\overset{\pi^m}{\to}R\to R/(\pi^m)\to 0, 
    \]
    we have 
    \[
    \Ext^1_R(R/(\pi^m),R)=R/(\pi^m).
    \]
    Thus the claim follows.
\end{proof}
\begin{proof}[Proof of Theorem \ref{degree_and_duality}]
    Let 
    \[
    0\to \bigoplus_{i=1}^r R(-\lambda_i)\to M\to Q\to 0
    \]
    be an exact sequence as Lemma \ref{fundamental_lemma}. Taking dual, we have the following short exact sequence:
    \[
    0\to M^\lor\to \bigoplus_{i=1}^r R(\lambda_i)\to \Ext^1_R(Q,R)\to 0.
    \]
    Thus, by the Proposition \ref{additivity_of_degree} and the Lemma \ref{degree_and_duality}, we have
    \[
    \deg(M^\lor)=\left(\sum_{i=1}^r \lambda_i\right)a_n(R)-\sum_{P}\length(H^0_P(Q))a_{n-1}(R/P)=-\deg(M)
    \]
\end{proof}
\section{Extension problem via valuations}
\subsection{Extensions}
In this section, we introduce the notion of $v$-valuative functions and geometric $v$-valuative functions. Let $R$ be a finitely generated domain over $\mathbbm{k}$. For a finitely generated valuation $v$ on $R$ centred on a closed point $x\in X=\Spec(R)$, we can construct a multi--graded degeneration $\mathfrak{X}:=\Spec(\mathcal{R})\to T(\sigma)$ as in the introduction. For a finitely generated torsion--free module $M$ on $R$, we can pull--back $E=M^\sim$ by the natural projection $\pi: \mathfrak{X}|_{(\mathbb{G}_m)^r}\cong X\times \mathbb{G}_m^r\to X$. We introduce the notion of geometric $v$-valuative function which gives an extension of $\pi^*E$ to a $\mathbb{G}_m^r$-equivariant sheaf on $\mathfrak{X}$ as the Rees module of $M$.
\begin{defn}\label{v_valuative_functions}
    Let $R$ be a finitely generated domain over $\mathbbm{k}$ and $M$ be a torsion--free module over $R$. For a finitely generated $\mathbb{R}$-valuation $v$ on $R$ centred at a $\mathbbm{k}$-valued point in $\Spec(R)$, a function $v_M: M\to \mathbb{R}\cup\{\infty\}$ is called a $v$-valuative function if 
    \begin{enumerate}
        \item $v_M(m_1+m_2)\geq \min\{v_M(m_1),v_M(m_2)\}$ for any $m_1,m_2\in M$.
        \item $v_M(am)=v(a)+v_M(m)$ for any $a\in R$, $m\in M$.
        \item $v_M(m)=\infty$ iff $m=0$.
        \item The associated graded module
        \[
        \gr_{v_M}(M):=\bigoplus_{\lambda\in \mathbb{R}}M_{\geq \lambda}/M_{>\lambda}
        \]
        is a finitely generated $\gr_v(R):=\bigoplus_{\lambda\geq 0}R_{\geq \lambda}/R_{>\lambda}$-module.
    \end{enumerate}
    The space of all $v$-valuative functions on $M$ is denoted as $\Val_{v,M}$. If the range of $v_M\in \Val_{v,M}$ is contained in $\mathbb{Z}\Gamma$, the submodule of $\mathbb{R}$ generated by $\Gamma:=v(R-\{0\})$, we say $v_M$ is {\bf geometric} and the set of all geometric $v$-valuative functions is denoted by $\Val_{v,M}^g$. 
\end{defn}
Note that if $v_M$ is a $v$-valuation, $\gr_{v_M}(M)$ is a torsion--free $\gr_v(R)$-module. For $v_M\in \Val_{v,M}^g$, we can define the Rees module $\Rees_{v_M}(M)$ as 
\[
\Rees_{v_M}(M):=\bigoplus_{s\in M(v)}M_st^{-s},
\]
where
\[
M_s=\{m\in M\mid v_M(m)\geq s\}.
\]

\begin{lem}
    Let $R$ be a finitely generated domain over $\mathbbm{k}$ and $M$ be a torsion--free module over $R$. Let $v\in \Val_{X,x}$ be a finitely generated $\mathbb{R}$-valuation on $R$ centred at a closed point $x\in \Spec(R)$. Then $\Val_{v,M}^g$ is a non--empty set. 
\end{lem}
\begin{proof}
    Since $M$ is torsion--free, we can embed $M$ into a free module $R^{\oplus N}$. Let $v_f$ be a valuation on $R^{\oplus N}$ defined as $v_f((a_1,\dots,a_N)):=\min_{1\leq i\leq N}\{v(a_i)\}$. Restriction of $v_f$ onto $M$ gives a geometric  $v$-valuation on $M$.
\end{proof}
\begin{defn}
    Let $R$ be a finitely generated $\mathbbm{k}$-domain and $M$ be a torsion--free $R$-module, and $v$ be a finitely generated $\mathbb{R}$-valuation centred at a $\mathbbm{k}$-valued point in $\Spec(R)$. The function 
    \[
    \Phi: \Val_{v,M}\to \mathbb{R}_{\geq 0}
    \]
    is defined as 
    \[
    \Phi(v_M)=\mu_{\max}-\mu_{\min},
    \]
    where $\mu_{\max}$ is defined as the slope of $\underline{M}_1$ and $\mu_{\min}$ the slope of $\underline{M}/\underline{M}_{l-1}$, where $(\underline{M}_i)_{i=0}^l$ is the Harder--Narasimhan filtration of $\gr_v(R)$-module $\underline{M}:=\gr_{v_M}(M)$.
\end{defn}
\begin{Rem}
    The functional $\Phi$ measures the instability of the graded module $\gr_{v_M}(M)$. In particular, $\Phi(v_M)=0$ if and only if $\gr_{v_M}(M)$ is semistable in the sense of Definition \ref{def_of_stability}.
\end{Rem}

\subsection{Hecke transform}
In this subsection, we consider modules on an affine domain $R$ over $\mathbbm{k}$ with finitely generated valuation $v\in \Val_{X,x}$ of rank $1$ (in short, we call such valuations {\bf quasi--regular}), where $x$ is a closed point of $X=\Spec(R)$. Let $\delta\in \mathbb{R}_{>0}$ be the index of $v$, which is defined by $\delta=\min_{f,g\in R\setminus \{0\}, v(f)\neq v(g)}|v(f)-v(g)|$. Let $M$ be a torsion--free $R$ module and $v_M\in \Val_{v,M}^g$ be a geometric $v$-valuative function. Note, in this case, that 
\[
\gr_{v_M}(M)=\bigoplus_{k\in \mathbb{Z}}M_{\geq _{\delta k}}/M_{\geq \delta(k+1)}
\]
holds. 
\begin{defn}
    Let $R$ be an affine domain over an algebraically closed field $\mathbbm{k}$ and $v\in \Val_{X,x}$ be a quasi--regular valuation centred at a closed point of $\Spec(R)$. Let $M$ be a torsion--free $R$-module and $v_M\in \Val^g_{v,M}$. Let $\underline{N}\subset \gr_{v_M}(M)$ be a submodule. The {\bf Hecke transform} (or {\bf elementary transform}) of $v_M$ along $\underline{N}$ is the function $v_M': M\to \delta\cdot \mathbb{Z}\cup\{\infty\}$ constructed as the following equation:
    \[
    v'_M(m)=\max\{k\delta\in \delta\cdot \mathbb{Z}\mid v_M(m)\geq k\delta\land p_{k\delta}(m)\in N_{k\delta}\},
    \]
    where $p_{k\delta}: M_{\geq k\delta}\to M_{\geq k\delta}/M_{\geq (k+1)\delta}$ is the quotient map.
\end{defn}
\begin{Rem}
    Let $\mathfrak{X}:=\Spec(\Rees_v(R))$ be the affine test configuration induced by $v$. Then we can construct a $\mathbb{G}_m$-equivariant sheaf $\tilde{E}=(\Rees_{v_M}(M))^\sim$ by $(M,v_M)$. As we see in the following, the above definition of Hecke transform gives the Hecke transform of $\tilde{E}$ along the submodule of $\tilde{E}|_{\mathfrak{X}_0}$ corresponding to $\underline{N}$. See \cite{CS_2018}, section 2.2.
\end{Rem}

In the above notation, let $M'_{\geq n\delta}\subset M$ be a submodule of $M$ defined by the following: 
\[
M'_{\geq n\delta}:=\{m\in M\mid v_M'(m)\geq n\delta\}=p_{n\delta}^{-1}((\underline{N})_{n\delta})+M_{\geq (n+1)\delta}.
\]
Then this makes the natural exact sequence
\[
0\to M'_{\geq n\delta}\to M_{\geq n\delta}\to \frac{M_{\geq n\delta}/M_{\geq (n+1)\delta}}{(\underline{N})_{n\delta}}\to 0,
\]
the third module is naturally isomorphic to 
\[
M_{\geq n\delta}/(M_{\geq (n+1)\delta}+p_{n\delta}^{-1}((\underline{N})_{n\delta})).
\]
Furthermore, we also have the following short exact sequence:
\[
0\to \frac{M_{\geq (n+1)\delta}/M_{\geq (n+2)\delta}}{(\underline{N})_{(n+1)\delta}}\to M'_{\geq n\delta}/M'_{\geq (n+1)\delta}\to (\underline{N})_{n\delta}\to 0,
\]
since the first injection is defined by
\[
M_{\geq (n+1)\delta}/(p_{(n+1)\delta}^{-1}((\underline{N})_{(n+1)\delta})+M_{\geq (n+2)\delta})\to (p_{n\delta}^{-1}((\underline{N})_{n\delta})+M_{\geq (n+1)\delta})/(p_{(n+1)\delta}^{-1}((\underline{N})_{(n+1)\delta})+M_{\geq (n+2)\delta}).
\]
Especially, the following proposition holds: 
\begin{prop}\label{the_fundamental_exact_seq_of_the_hecke_transforms}
    Let $R$ be an affine domain over an algebraically closed field $\mathbbm{k}$ and $v\in \Val_{X,x}$ be a quasi--regular valuation centred at a closed point $x$ of $\Spec(R)$. Let $M$ be a torsion--free $R$ module and $v_M\in \Val^g_{v,M}$. Let $\underline{N}$ be a submodule of $\gr_{v_M}(M)=:\underline{M}$ and $v_M'$ be the Hecke transformation of $v_M$ along $\underline{N}$. We have the following exact sequence of $\gr_v(R)$-modules:
    \[
    0\to (\underline{M}/\underline{N})(\delta)\to \underline{M'}\to \underline{N}\to 0,
    \]
    where $\underline{M'}=\gr_{v_M'}(M)$. In particular, $v_M'\in \Val^g_{v,M}$ if $\underline{N}$ is a saturated submodule of $\underline{M}$.
\end{prop}
Thus, we get the following proposition:
\begin{prop}\label{descend_the_energy}
    Let $R$, $x$, $v$, $M$, $v_M$ as the above proposition. Let $(\underline{M_i})_{i=1}^l$ be the HNF of $\underline{M}$. Let $v_M'$ be the Hecke transform of $v_M$ along $\underline{M_1}$. Then we get the following inequality: 
    \[
    \Phi(v_M')\leq \max\{\mu(\underline{M_2}/\underline{M_1})-\mu_{\min}(\underline{M}),\mu(\underline{M_2}/\underline{M_1})-\mu_{\max}(\underline{M})+\delta,\Phi(v_M)-\delta\}
    \]
    In particular, if $\Phi(\underline{M'})\geq \delta$, we get $\Phi(v_M')<\Phi(v_M)$.
\end{prop}
\begin{proof}
    The proof is the same as that of \cite[Lemma~3.2]{CS_2018}. We get the following short exact sequence: 
    \[
    0\to (\underline{M}/\underline{M_1})(\delta)\to \underline{M'}\to \underline{M_1}\to 0.
    \]
    Let $(\underline{M'_j})_{j=1}^{l'}$ be the HNF of $\underline{M'}$. Then, we have
    \[
    \mu(\underline{M'_1})\leq \max\{\mu(\underline{M_2}/\underline{M_1})+\delta,\mu(\underline{M_1})\}
    \]
    by the Proposition \ref{sea-saw-lemma}. On the other hand, We have an exact sequence
    \[
    0\to \underline{Q_1}\to \underline{M'}/\underline{M'_{l'-1}}\to \underline{Q_2}\to 0
    \]
    for some quotients $(\underline{M}/\underline{M_1})(\delta)\to \underline{Q_1}$, $\underline{M_1}\to \underline{Q_2}$. Thus we have 
    \[
    -\mu_{\min}(\underline{M'})\leq \max\{-\mu_{\min}(\underline{M})-\delta,-\mu(\underline{M_1})\}.
    \]
    Therefore, we have
    \[
    \Phi(v_M')\leq \max\{\mu(\underline{M_2}/\underline{M_1})-\mu_{\min}(\underline{M}),\mu(\underline{M_2}/\underline{M_1})-\mu_{\max}(\underline{M})+\delta,\Phi(v_M)-\delta\}
    \]
\end{proof}
In this setting, we can show that $a_n(R)\in \delta^{-n}\cdot \mathbb{Z}$. Thus we get 
\[
\deg(\gr_{v_N}(N))=\sum_{P}\length(H^0_P(Q))a_{n-1}(R/P)-(\sum \lambda_i)a_n(R)\in \delta^{-(n-1)}\mathbb{Z}
\]
for any finite torsion--free module $N$ and $v_N\in \Val_{v,N}$ with $v_N(N\setminus \{0\})\subset \delta\cdot\mathbb{Z}$, or this can be shown by using the usual Hilbert--Samuel theory since $\dim(M_{\leq \delta n})$ is quasi-polynomial in $n$. Therefore, 
\[
\mu(\gr_{v_N}(N))\in \frac{1}{r!a_n(R)\delta^{n-1}}\mathbb{Z}
\]
for any $(N,v_N)$ with $\rank(N)\leq r$, $v_N(N\setminus \{0\})\subset \delta\cdot \mathbb{Z}$. In particular, $\Phi(v_M)$ takes discrete value if $v$ is quasi--regular. Combining this with the Proposition \ref{descend_the_energy}, we get the following theorem:
\begin{Thm}
    Let $R$ be a affine domain over an algebraically closed field $\mathbbm{k}$ and $v\in \Val_{X,x}$ be a quasi--regular valuation centred at a closed point $x$ of $\Spec(R)$, with index $\delta$. Let $M$ be a torsion--free $R$ module. Then there exists $v_M\in \Val^g_{v,M}$ such that $\Phi(v_M)\in [0,\delta)$.
\end{Thm}

\begin{defn}
    Let $R$ be an affine domain over an algebraically closed field $\mathbbm{k}$ and $v\in \Val_{x,X}$ be a quasi--monomial valuation centred at a closed point $x\in \Spec(R)$ with index $\delta$. Let $M$ be a torsion--free $R$ module. $v_M\in \Val^g_{v,M}$ is called {\bf optimal} if $\Phi(v_M)\in [0,\delta)$. 
\end{defn}
\subsection{Uniqueness of optimal valuative function}
In this subsection, we prove the following theorem: 
\begin{Thm}\label{uniqueness}
    Let $R$ be a affine domain over an algebraically closed field $\mathbbm{k}$ and $v\in \Val_{x,X}$ be a quasi--monomial valuation centred at a closed point $x\in X:=\Spec(R)$ with index $\delta$. Let $M$ be a torsion--free $R$ module and $v_M,v_M'\in \Val^g_{v,M}$ be two optimal $v$--valuative functions. Then the following statements hold:
    \begin{enumerate}
        \item If $\Phi(v_M)+\Phi(v_M')<1$, then there is $c\in \delta\mathbb{Z}$ such that $v_M'=v_M+c$.
        \item Otherwise, there exists $\underline{M'_{k'}}\subset \underline{M'}=\gr_{v_M'}(M)$ appearing in the HNF of $\underline{M'}$ such that $v_M$ is the parallel transport of the Hecke transform of $v_M$ along $\underline{M'_{k'}}$. 
    \end{enumerate}
\end{Thm}

Without loss of generality, we can assume that $\delta=1$. Thus we prove in this case. Let $\Rees_{v}(R)$ be the Rees algebra of $R$ for this valuation $v$, i.e.
\[
\Rees_v(R)=\bigoplus_{k\in \mathbb{Z}}R_kt^{-k},
\]
where $R_k:=R_k(v):=\{f\in R\mid v(f)\geq k\}$. Note that $\Rees_v(R)$ is not normal in general. Let $\tilde{M}=\Rees_{v_M}(M), \tilde{M'}=\Rees_{v_M'}(M)$ be the Rees modules, i.e.
\[
\tilde{M}=\bigoplus_{k\in \mathbb{Z}}M_kt^{-k}, \tilde{M'}=\bigoplus_{k\in \mathbb{Z}}M'_kt^{-k},
\]
where
\[
M_k:=\{m\in M\mid v_M(m)\geq k\}, M'_k:=\{m\in M\mid v_M'(m)\geq k\}.
\]
Since $v_M$ and $v_M'$ are $v$-valuation, $\tilde{M}$ and $\tilde{M'}$ are finitely generated $\Rees_v(R)$--module. This statement can be shown by standard Artin--Rees argument. Note also that $v_M$ can be recovered by $\tilde{M}$, simply $v_M(m)=\max\{k\in \mathbb{Z}\mid mt^{-k}\in \tilde{M}\}$. 

Let $\pi: \mathfrak{X}:=\Spec(\Rees_v(R))\to \mathbb{A}^1$ be the natural projection. Then $\pi$ is $\mathbb{G}_m$-equivariant morphism and there is a natural $\mathbb{G}_m$-equivariant isomorphism $\mathfrak{X}|_{\mathbb{A}^1-\{0\}}=\pi^{-1}(\mathbb{A}^1-\{0\})\cong X\times \mathbb{G}_m$, i.e. $\pi: \mathfrak{X}\to \mathbb{A}^1$ is an affine test configuration.
Let $\tilde{E}, \tilde{E'}$ be coherent sheaves on $\mathfrak{X}$ defined by $\tilde{M}, \tilde{M'}$ respectively. Then $\tilde{E}$, $\tilde{E'}$ endow the natural $\mathbb{G}_m$-actions. Furthermore, there are the natural isomorphisms
\[
\rho:\tilde{E}|_{\mathfrak{X}_{\mathbb{A}^1-\{0\}}}\cong p_1^*E\cong \tilde{E'}|_{\mathfrak{X}_{\mathbb{A}^1-\{0\}}},
\]
where $E$ is the sheaf on $X$ associated by $M$ and $p_1: \mathfrak{X}_{\mathbb{A}^1-\{0\}}=X\times \mathbb{G}_m\to X$ is the first projection. In this setting, $\rho$ can be seen as the section of $\sHom_{\mathcal{O}_\mathfrak{X}}(\tilde{E},\tilde{E'})$ on $\mathfrak{X}-\mathfrak{X}_0$. Now $D:=\mathfrak{X}_0$ is $\mathbb{G}_m$-invariant Cartier divisor on $\mathfrak{X}$. Thus there is $c\in \mathbb{Z}$ such that $\rho$ extends to $\tilde{\rho}: \tilde{E}\to \tilde{E'}(cD)$ and $\tilde{\rho}|_D\neq 0$. Now, $\tilde{E'}(cD)$ is a $\mathbb{G}_m$--equivariant sheaf on $\mathfrak{X}$ associated by $t^{-c}\tilde{M'}$, i.e. associated by the $v$-valuative function $v_M'+c$. Re--taking $v_M'$ by $v_M'+c$, we may assume that $c=0$. Similarly, we can take some $l\in \mathbb{Z}$ such that $\rho^{-1}$ extends to $\tilde{\rho}^{-1}: \tilde{E'}\to \tilde{E}(lD)$ and $\tilde{\rho}^{-1}|_D\neq 0$. Note that the existence of the extension $\tilde{\rho}^{-1}$ means that if $v_M'(m)\geq a$, we always have $v_M(m)\geq a-l$. Thus, in this case we have for all $m\in M$, $v_M'(m)\geq v_M(m)\geq v_M'(m)-l$. This means that $l$ must be non--negative. Let $(\underline{M_i})_{i=0}^N$, $(\underline{M'_j})_{j=0}^{N'}$ be the HNFs of $\underline{M}:=\gr_{v_M}(M)$, $\underline{M'}:=\gr_{v_M'}(M)$ respectively. Let $k$ be the maximal number with property $\tilde{\rho}|_D(M_k)=0$. Then there is a non-zero map
\[
\tilde{\rho}_D: \underline{M_{k+1}}/\underline{M_k}\to \underline{M'}, 
\]
i.e. $\mu_{\min}:=\mu_{\min}(v_M)\leq \mu_{k+1}:=\mu(\underline{M_{k+1}}/\underline{M_k})\leq \mu'_{\max}=\mu_{\max}(v_M')$. Similarly, we have $\mu'_{\min}\leq \mu_{\max}-l$. Thus we have
\[
0\leq l\leq \Phi(v_M)+\Phi(v_M')<2
\]
Thus, there are two cases: $l=0$ or $l=1$. 
\begin{enumerate}
    \item Consider the case when $l=0$. Note firstly that $\tilde{\rho}^{-1}\circ \tilde{\rho}: \tilde{E}\to \tilde{E}(lD)$ is the multiplication by $t^l$. In fact, if $mt^{-d}\in \tilde{M}$, since $\rho: \tilde{M}[t^{-1}]\to \tilde{M'}[t^{-1}]$ is identity map, we have $\tilde{\rho}(mt^{-d})=mt^{-d}$. If $m't^{-d}\in \tilde{M'}$, the same reasoning shows that $\tilde{\rho}(m't^{-d})=m't^{-d}\in \tilde{M}(l)[t^{-1}]$, and the only possibility of elements in $\tilde{M}(lD)$ whose restriction to $\mathfrak{X}|_{\mathbb{A}^1-\{0\}}$ is $m't^{-d}$ is $m't^{-(d-l)}\otimes t^{-l}(\in M_{d-l}t^{-(d-l)}\otimes t^{-l})$ itself, since $\tilde{M}(lD)$ is torsion--free. This shows that $\tilde{\rho}^{-1}\circ \tilde{\rho}(mt^{-d})=mt^{-(d-l)}\otimes t^{-l}$. If $l=0$, Then $\tilde{\rho}^{-1}\circ \tilde{\rho}$ is identity map, and so is $\tilde{\rho}\circ \tilde{\rho}^{-1}$. Thus, $\tilde{\rho}$ is an isomorphism. Note that this shows the (i) of the Theorem \ref{uniqueness}.
    \item Next, let us consider the case $l=1$. Since $\tilde{\rho}: \tilde{E}\to \tilde{E'}$ is injective, we treat $\tilde{E}$ as a subsheaf of $\tilde{E'}$. Note that, in this case, $v_M+1\geq v_M'\geq v_M$ holds. Thus, we have
    \begin{align*}
        \tilde{M'}/\tilde{M}=&\bigoplus_d (M'_d/M_d)t^{-d}\\
        =&\bigoplus_d \frac{M'_d/M'_{d+1}}{M_d/M'_{d+1}}t^{-d}\\
    \end{align*}
    Define $\underline{N}:=\bigoplus_d M_d/M'_{d+1}\subset \underline{M'}$. Then it is easy to show that $v_M$ is the Hecke transform of $v_M'$ along $\underline{N}$. Therefore, it suffices to show that $\underline{N}$ appears in the HNF of $\underline{M'}$.
    
    Let $\underline{M'_{k'}}$ be the maximal subsheaf appearing in the HNF of $\underline{M'}$ such that $\tilde{\rho}^{-1}|_D(\underline{M'_{k'}})=0$. Then we have $\Image(\tilde{\rho}|_D)= \underline{M'_{k'}}$. In fact, if $\Image(\tilde{\rho}|_D)\not\subset \underline{M'_{k'}}$, we have non--trivial morphism
    \[
    \underline{M}\to \underline{M'}/\underline{M'_{k'}}
    \]
    induced by $\tilde{\rho}|_D$. Thus we have $\mu_{\min}\leq \mu'_{k'+1}$. Otherwise, there is a non--trivial morphism
    \[
    \underline{M'_{k'+1}}/\underline{M'_{k'}}\to \underline{M}(-1)
    \]
    induced by $\tilde{\rho}^{-1}|_D$ and we have $\mu'_{k'+1}\leq \mu_{\max}-1$. However, these inequalities shows that 
    \[
    \Phi(v_M)\geq \mu'_{k'+1}+1-\mu_{\min}\geq 1.
    \]
    This contradicts with the assumption $\Phi(v_M)<1$. Therefore, we have $\Image(\tilde{\rho}|_D)\subset \underline{M'_{k'}}$.

    On the other hand, $\Image(\tilde{\rho}|_D)=\underline{N}$ by definition, and this is also the kernel of $\tilde{\rho}^{-1}|_D$. In fact, observation in (i) shows that $\rho^{-1}_D: \underline{M'}\to \underline{M}(-1)$ is given by natural inclusions $M'_d\subset M_{d-1}$. Thus, we have 
    \[
    \ker(\tilde{\rho}^{-1}|_D)=\bigoplus_{d}M_d/M'_{d+1}=\underline{N}.
    \]
    Therefore, we have $\underline{N}=\ker(\tilde{\rho}|_D)\subset \underline{M'_{k'}}$. It is true by definition that $\underline{M'_{k'}}\subset \ker(\tilde{\rho}|_D)$. As a result, we get (ii) of the Theorem\ref{uniqueness}.\qed
\end{enumerate}
As a consequence, we have the following theorem: 
\begin{Thm}\label{main_theorem}
    Let $R$ be an affine domain over an algebraically closed field $\mathbbm{k}$ and $v\in \Val_{X,x}$ be a quasi--regular valuation centred at a closed point $x\in X:=\Spec(R)$ with index $\delta$. For any torsion--free $R$--module $M$, the following statements hold: 
    \begin{enumerate}
        \item There exists $v_M\in \Val_{v,M}^g$ such that $\Phi(v_M)\in [0,\delta)$.
        \item The graded module of the HNF of $\gr_{v_M}(M)$ is independent of the choice of $v_M\in \Val_{v,M}^g$ with $\Phi(v_M)\in [0,\delta)$ up to twisting the grading of each direct summand.
    \end{enumerate}
\end{Thm}
\begin{Rem}
    This theorem can be viewed as an algebraic generalization of the Theorem 1.4 (I)-(III) of Chen--Sun \cite{CS_2018}. In fact, if $X$ is smooth at $x$ and $v\in \Val_{X,x}$ is the blow--up valuation, we have the existence of $v_M\in \Val_{v,M}^g$ such that $\Phi(v_M)<1$ by Theorem \ref{main_theorem} (i), and $\gr^{HNF}(\gr_{v_M}(M))$ is independent of the choice of such $v$-valuative functions up to twisting each direct summand. 
\end{Rem}
\section{Examples}
We conclude this paper with simple examples illustrating the main theorem. First example shows an example of optimal algebraic tangent cone on a smooth point along a valuation which is not a blow-up valuation. Second example shows examples of optimal algebraic tangent cone on singular points along blow-up valuations.
\subsection{Optimal algebraic tangent cone on affine plane along monomial valuations}
Let $X=\mathbb{A}^2_\mathbbm{k}=\Spec(\mathbbm{k}[x,y])$ be the affine plane over $\mathbbm{k}$. Let $v:\mathbbm{k}[x,y]\to \mathbb{R}\cup \{\infty\}$ be the monomial valuation defined by $v(x)=1$, $v(y)=2$. Consider the extension problem of $T_X\cong \mathcal{O}_X^{\oplus 2}=(\mathbbm{k}[x,y]^{\oplus 2})^\sim$. 

Firstly, since it is trivial vector bundle, we can define the trivial $v$-valuative function $v_0: \mathbbm{k}[x,y]^{\oplus 2}\to \mathbb{R}\cup \{\infty\}$ defined by $v_0(e_1)=v_0(e_2)=0$, where $e_1=1\oplus 0,e_2=0\oplus 1$. The associated graded module over $\gr_v(\mathbbm{k}[x,y])=\mathbbm{k}[x^{(1)}, y^{(2)}]$ is $\gr_v(\mathbbm{k}[x,y])^{\oplus 2}$. 

However, since $\mathcal{O}_X^{\oplus 2}$ is isomorphic to $T_X$ via $\partial_x\mapsto e_1$, $\partial_y\mapsto e_2$, the action of $\mathbb{G}_m$ on $X$ induced by the valuation $v$ gives another $v$-valuative function $v_1: \mathbbm{k}[x,y]^{\oplus 2}\to \mathbb{R}\cup \{\infty\}$ defined by $v_1(e_1)=1$, $v_1(e_2)=2$, such that the associated graded module is $\gr_v(\mathbbm{k}[x,y])(-1)\oplus \gr_v(\mathbbm{k}[x,y])(-2)$. The HNF of this module is given by 
\[
0\subset \gr_v(\mathbbm{k}[x,y])(-1)\subset \gr_v(\mathbbm{k}[x,y])(-1)\oplus \gr_v(\mathbbm{k}[x,y])(-2),
\]
and the Hecke transform of $v_1$ along $\gr_v(\mathbbm{k}[x,y])(-1)$ gives the valuation $v_0+1$. Thus, in this case, one can verify the uniqueness of the optimal algebraic tangent cone explicitly.
\subsection{Optimal algebraic tangent cone on singular points}
Let $C$ be a smooth projective curve of genus $g$ over $\mathbb{C}$ and $L$ be a very ample line bundle on $C$. Let $\phi: C\to \mathbb{P}(H^0(C,L))=\Proj(\Sym(H^0(C,L)^\lor))$ be the closed embedding defined by the complete linear series $|L|$. Let $ \mathbf{H}^0(C,L):=\Spec(\Sym(H^0(C,L)^\lor))$ and $\pi: U:=\mathbf{H}^0(C,L)-\{0\}\to \mathbb{P}(H^0(C,L))$ be the natural projection. Define $X$ as the closure of $\pi^{-1}(\phi(C))$ in $\mathbf{H}^0(C,L)$. Then $X$ is naturally isomorphic to $\Spec(R)$, where $R=\bigoplus_{m\geq 0}H^0(C,mL)$, and $X$ has a natural $\mathbb{G}_m$-action. Now, $\pi: U\to \phi(C)\cong C$ is principal $\mathbb{G}_m$-bundle corresponding to the line bundle $L^{-1}$. Thus, we have the following $\mathbb{G}_m$-equivariant exact sequence: 
\[
0\to \mathcal{O}_U\to T_U\overset{d\pi}{\to} \pi^*T_C\to 0.
\]
This short exact sequence descends to a short exact sequence on $C$, whose extension class is $-2\pi\sqrt{-1}c_1(L)\in H^1(C,\Omega_C)$:
\begin{equation}\label{Euler_sequence_of_curves}
    0\to \mathcal{O}_C\to E\to T_C\to 0.
\end{equation}
Let $v$ be a valuation on $X$ induced by the $\mathbb{G}_m$-action, $T$ be a graded module corresponding to the tangent sheaf $T_X$ with $\mathbb{G}_m$-action induced by the action on $X$, and $v_T$ be the $v$-valuative function on $T$ induced by the grading of $T$. Then $H^0(C,E(k))=T_k$ for $k\in \mathbb{Z}$. Thus, the stability of $E$ is equivalent to the stability of the graded module $T$.
\begin{itemize}[leftmargin=0pt]
    \item Firstly, consider the case when $g=0$. In this case, $X$ is smooth at the vertex $0$, and the exact sequence \ref{Euler_sequence_of_curves} is nothing but the Euler sequence of $\mathbb{P}^1$ and this shows that $T$ is semistable, i.e. $v_T$ is optimal.
    \item Secondly, consider the case when $g=1$. In this case, $X$ has a log--canonical singularity at the vertex $0$ of $X$. The exact sequence \ref{Euler_sequence_of_curves} is the unique nontrivial extension of $\mathcal{O}_C$ by itself. This shows that $v_T$ is optimal. 
    \item Finally, consider the case when $g\geq 2$. In this case, $X$ has a non--log--canonical singularity at the vertex $0$ of $X$. $E$ is unstable and the exact sequence \ref{Euler_sequence_of_curves} gives the HNF of $E$. Thus, to make an optimal $v$-valuative function, we execute processes of Hecke transforms along maximal destabilizing submodules. Before doing this procedure, we prove the following proposition: 
    \begin{prop}
        Let $R=\bigoplus_{k\geq 0}R_k$ be a graded ring of finite type on $R_0=\mathbbm{k}$, and $M=\bigoplus_{k\in \mathbb{Z}} M_k$ be a finitely generated graded $R$-module. Put $v_M$ as a valuation associated to the grading of $M$. For a saturated graded submodule $N=\bigoplus_{k\in \mathbb{Z}}N_k\subset M$, the Hecke transform of $v_M$ along $N$ is the valuation whose associated graded module is $N\oplus(M/N)(1)$.
    \end{prop}
    \begin{proof}
        Let $v_M'$ be the Hecke transform of $v_M$ along $N$. Then we have $M_{v_M'\geq k}=M_{\geq k+1}+N_{\geq k}=N_k\oplus M_{\geq k+1}$, where $M_{\geq k}=\bigoplus_{l\geq k}M_l$ and $N_{\geq k}=\bigoplus_{l\geq k}N_l$. Thus, we have 
        \[
        \gr_{v_M'}(M)=\bigoplus_{k\in \mathbb{Z}}\frac{N_k\oplus M_{k+1}\oplus M_{\geq k+2}}{N_{k+1}\oplus M_{\geq k+2}}=N\oplus (M/N)(1).
        \]
    \end{proof}
    By construction, this direct sum decomposition is consistent with the exact sequence \ref{the_fundamental_exact_seq_of_the_hecke_transforms}, which is the following short exact sequence in this case:
    \[
    0\to (M/N)(1)\to \gr_{v_M'}(M)\to N\to 0.
    \]
    If we do the Hecke transform of $v_M'$ again by $N$, we have the valuation $v_M^{(2)}$ defined by 
    \[
    M_{v_M^{(2)}\geq k}=N_k\oplus N_{k+1}\oplus M_{\geq k+2},
    \]
    Thus, we have
    \[
    \gr_{v_M^{(2)}}(M)=N\oplus (M/N)(2).
    \]
    Doing this procedure inductively, we have $v$-valuative functions $v_M^{(l)}$ ($l\in \mathbb{Z}_{>0}$) such that 
    \[
    \gr_{v_M^{(l)}}(M)=N\oplus (M/N)(l).
    \]
    Applying this process to our setting, we have an optimal $v$-valuative function $v_T'$ such that 
    \[
    \gr_{v_T'}(T)=R\oplus (T/R)\left(\left\lfloor\frac{2g-2}{\deg(L)}\right\rfloor\right)
    \]
    Since $T/R$ is rank $1$ graded module with $\mu(T/R)=(2-2g)/\deg(L)$, we have that $v_T'$ is an optimal $v$-valuative function.
\end{itemize}
\section*{Acknowledgements}
The author would like to express his sincere gratitude to Yuji Odaka for his constant guidance, encouragement, and many insightful discussions throughout this work. He is also grateful to Eiji Inoue for valuable advice and for pointing out several related references. The author thanks Song Sun for taking the time to discuss this project during his visit to Japan and for his encouraging comments. 

The author also thanks the members of the algebraic geometry group at Kyoto University for helpful discussions and questions, in particular during seminar presentations related to this work.
\bibliographystyle{plain}
\bibliography{reference}
\vspace{5mm} \footnotesize \noindent
Contact: {\tt hada.yohei.63n@st.kyoto-u.ac.jp} \\
Department of Mathematics, Kyoto University, Kyoto 606-8285. %JAPAN \\

\end{document}